\documentclass[]{amsart}

\usepackage{amssymb}
\usepackage{amsbsy}
\usepackage{amscd}
\usepackage{amsmath,rotating}
\usepackage{amsthm}
\usepackage{graphicx,amssymb,amsfonts,amsmath,amsthm,newlfont}
\usepackage{epsfig}

\usepackage[all,2cell]{xy} \UseAllTwocells \SilentMatrices

\newtheorem{thm}{Theorem}[section]

\newtheorem{lem}[thm]{Lemma}

\theoremstyle{definition}
\newtheorem{defn}[thm]{Definition}
\newtheorem{conj}{Conjecture}

\newtheorem{ex}[thm]{Examples}
\newtheorem{example}[thm]{Example}
\theoremstyle{remark}
\newtheorem{rem}[thm]{Remark}
\numberwithin{equation}{section}
\newcommand{\rin}{\begin{sideways}$\in$\end{sideways}}

\newcommand{\Z}{\mathbb Z}
\newcommand{\C}{\mathbb C}

\newcommand{\R}{\mathbb R}
\newcommand{\N}{\mathbb N}

\newcommand{\NN}{\N_+}
\newcommand{\Lo}{\mathcal{L}}
\newcommand{\Homt}{\Ho^{\MT_+}}

\font \rus= wncyr10
\newcommand{\sha}{\, \hbox{\rus x} \,}

\newcommand{\Ho}{\mathcal{H}}
\newcommand{\Amt}{\mathcal{A}^{\MT}}

\newcommand{\zetam}{\zeta^{ \mathfrak{m}}}
\newcommand{\zetaa}{\zeta^{ \mathfrak{a}}}
\newcommand{\zetal}{\zeta^{ \mathfrak{L}}}

\newcommand{\Imot}{I^{\mathfrak{m}}}
\newcommand{\Imota}{I^{\mathfrak{a}}}
\newcommand{\Imotl}{I^{\mathfrak{L}}}

\newcommand{\Q}{\mathbb Q}

\newcommand{\To}{\longrightarrow}

\newcommand{\G}{\mathbb{G}}

\newcommand{\Ao}{\mathcal{A}}
\newcommand{\tdots}{.\, .\,}

\newcommand{\MT}{\mathcal{MT}}
\newcommand{\U}{\mathcal{U}}
\newcommand{\Gu}{\mathcal{G_U}}

\def\0skip{\vskip 2pt}
\def\1skip{\vskip 4pt}
\def\2skip{\vskip 10pt}
\def\3skip{\vskip 0.15in}
\def\4skip{\vskip 0.20in}
\def\5skip{\vskip 0.25in}
\def\6skip{\vskip 0.30in}

\author[F. Brown]{Francis C. S. Brown}

\address{}

\email{brown@math.jussieu.fr}



\begin{document}
\begin{title}[On the decomposition of motivic multiple zeta values]{On the decomposition of motivic multiple zeta values}\end{title}
\maketitle
\begin{abstract}  We review motivic aspects of multiple zeta values, and  as an application,  we give an exact-numerical algorithm to decompose any (motivic) multiple zeta value of given weight into a chosen  basis up to that weight.
\end{abstract}
\section{Introduction}

The  aim of these notes is  to present motivic aspects of multiple zeta values in  concrete  terms, and give applications which might be of  use to  physicists.
Most introductory texts on multiple zeta values focus exclusively on the relations they   satisfy. Here,  we take the opposite point of view, and put the emphasis on 
the coalgebra structure underlying the  motivic multiple zeta values. There are two applications:
\begin{enumerate}
\item we show how to use the coalgebra structure to  decompose any multiple zeta value numerically into a conjectural basis.
\item we show how to lift certain identities between multiple zeta values, i.e., real numbers, to their motivic versions.
\end{enumerate}
The first point  requires  explanation. Since the $\Q$-vector space of multiple zeta values is finite-dimensional in each weight,  standard lattice reduction algorithms give a numerical  way to write an arbitrary multiple zeta value of given weight in terms of some  chosen spanning set. The point of  $(1)$ is that the coalgebra structure enables one to replace this single  high-dimensional lattice reduction problem with a sequence of  one-dimensional lattice reductions.  This is simply the problem of  identifying a   rational number  $\alpha\in \Q$  which is  presented as an element $\alpha \in \R$   to arbitrarily high accuracy, and  
can be  done using continued fractions. In fact, we expect   that 
there exists a relatively small a priori bound on the denominators of the  rational numbers $\alpha$ which can arise, and  so this algorithm   should be workable in  practice.

 An application of $(2)$  might be  to prove that certain families of relations between multiple zeta values are  `motivic'. 
 The idea behind this    was  used for the main theorem of \cite{Br23}, 
 where one had to lift a certain relation between actual multiple zeta values  to their motivic versions.

The paper is set out as follows. In \S2, we review some basic properties of iterated integrals for motivation. In $\S3$  we briefly review the structure of  the category of mixed Tate motives over $\Z$ and state the main properties of motivic multiple zeta values.  In $\S4$ we show how to define certain derivation operators $\partial^{\phi}_{2k+1}$, where $k\geq 1$, which act on the space of motivic multiple zeta values.  In \S5  we  
describe the decomposition algorithm $(1)$ using these operators,  and in \S6 we provide a worked example of this algorithm. 
 The reader who is only interested in implementing the algorithm may turn immediately to $\S\S5.1-5.2$, 
 which can be read independently from the rest of the paper.
  
\section{Iterated Integrals}  We begin with some generalities on iterated integrals, before specializing to the case of iterated integrals on the punctured projective line.

\subsection{General iterated integrals.}\label{sectGenitin} Let $M$ be a smooth $C^{\infty}$  manifold over $\R$, and let $k$ be the real or complex numbers.  Let $\gamma: [0,1] \rightarrow M$ be a piecewise smooth path on $M$, 
and  let $\omega_1,\ldots, \omega_n$ be 
smooth $k$-valued 1-forms on $M$. Let us  write
$$\gamma^*(\omega_i) = f_i(t) dt\ ,$$
for the pull-back of the forms $\omega_i$ to the interval $[0,1$].

\begin{defn} Let the iterated integral of $\omega_1,\ldots,\omega_n$ along $\gamma$ be
\begin{equation} \label{defitint} \int_{\gamma} \omega_1\ldots \omega_n = \int_{0\leq t_1\leq \ldots \leq t_n\leq 1} f_1(t_1) dt_1 \ldots f_n(t_n) dt_n\ . 
\end{equation}
More generally, an iterated integral is any $k$-linear combination of such integrals. The empty  integral ($n=0$) is defined to be the constant  $1$.
\end{defn}
The   iterated integrals $\int_{\gamma} \omega_1\ldots \omega_n$ do not depend on the choice of parametrization of the path $\gamma$, and  satisfy the following   basic properties:
 \1skip
 
\emph{Shuffle product formula}. Given $1$-forms $\omega_1,\ldots, \omega_{r+s}$ one has:
$$\int_{\gamma} \omega_1\ldots\omega_r \int_{\gamma} \omega_{r+1}\ldots \omega_{r+s} =\sum_{\sigma \in \Sigma(r,s)}  \int_{\gamma} \omega_{\sigma(1)} \ldots \omega_{\sigma(n)}\ ,$$
where $n=r+s$, and  $\Sigma(r,s)$ is the set $(r,s)$-shuffles:
$$\Sigma(r,s) = \{\sigma\in \Sigma(n): \sigma(1)<\ldots<\sigma(r) \hbox{ and } \sigma(r+1)<\ldots<\sigma(r+s)\}\ .$$
As a general rule, for any letters $a_1,\ldots, a_{r+s}$,  we shall formally write
\begin{equation}\label{shuffdef} a_1\ldots a_r \sha a_{r+1}\ldots a_{r+s} = \sum_{\sigma \in \Sigma(r,s)} a_{\sigma(1)} \ldots a_{\sigma(r+s)} \ ,
\end{equation}
viewed  in  $\Z\langle a_1,\ldots, a_{r+s}\rangle$, the free $\Z$-module spanned by  words in the $a$'s.
\1skip

\emph{Composition of paths}. If $\alpha,\beta:I\rightarrow M$ are two  piecewise smooth paths such that $\beta(0)=\alpha(1)$, then let $\alpha\beta$ denote the composed path obtained by traversing first $\alpha$ and then $\beta$.  
Then  
$$\int_{\alpha \beta} \omega_1\ldots\omega_n =\sum_{i=0}^n  \int_{\alpha} \omega_1\ldots \omega_i \int_{\beta} \omega_{i+1}\ldots \omega_{n}\ ,$$
 where  recall that the empty iterated integral ($n=0$) is just the constant  $1$.
\1skip

\emph{Reversal of paths}.   If $\gamma^{-1}(t)=\gamma(1-t)$ denotes the reversal of the path $\gamma$, then  we have the following reflection formula:
$$\int_{\gamma^{-1}} \omega_1\ldots\omega_n = (-1)^n \int_{\gamma} \omega_n\ldots \omega_1\ .$$
\1skip

\emph{Functoriality}.   If $f:M'\rightarrow M$  is a smooth map, and $\gamma:[0,1]\rightarrow M'$ a piecewise smooth path, then we have: 
$$\int_{\gamma} f^*\omega_1\ldots f^*\omega_n =  \int_{f(\gamma)} \omega_1\ldots \omega_n \ .$$

\subsection{The punctured projective line.} Now let us consider the case where   $k=\C$,   $S$ is a finite set of  points in $\C$, and 
$M=\C\backslash S$. Consider the set of closed one forms
\begin{equation}\label{oneforms} {dz\over z-a_i} \in \Omega^1(M)\ \hbox{ where } a_i \in S \ .
\end{equation}
Let $a_0,a_{n+1} \in M$ and let $\gamma$ be a path with endpoints  $\gamma(0)=a_0, \gamma(1)=a_{n+1}$.  Using the notation from \cite{GG}, set:\
\begin{equation}\label{Igamma}
I_{\gamma}(a_0;a_1,\ldots, a_n;a_{n+1}) = \int_{\gamma} {dz \over z-a_1} \ldots {dz \over z-a_n}\ . 
\end{equation} 
Since the exterior product of any two  forms $(\ref{oneforms})$ is zero and each one is  closed, one can  show that the iterated integrals $(\ref{Igamma})$ only depend on the homotopy class of $\gamma$ relative to its endpoints.
When the path $\gamma$ is clear from the context, it can be dropped from the notation.

A variant is to take the limit points $a_0,a_{n+1}$ in the set  $S$, in which case only the interior of $\gamma([0,1])$ lies in $M$. When  the integral $(\ref{Igamma})$ converges,
we can extend the definition to this case 
and show that the basic properties of \S\ref{sectGenitin} still hold. Even when it does not converge, $(\ref{Igamma})$ can   be defined by a suitable logarithmic regularization procedure (tangential basepoint).

\subsection{Multiple zeta values.}  From now on, we shall only consider the case where $M = \C\backslash \{0,1\}$, and thus  all  $a_i\in \{0,1\}$. 
 There  is a canonical path  $\gamma: (0,1)\rightarrow M$ where $\gamma(t) =t$, but note that the endpoints of $\gamma$ no longer lie in $M$.
Write
\begin{eqnarray} \label{rhodef} \rho: \NN^r  &\To& \{0,1\}^\times \\
\rho(n_1,\ldots, n_r) &=& 10^{n_1-1} \ldots 10^{n_r-1} \nonumber 
\end{eqnarray}
where   $0^k$ denotes a sequence of $k$ zeros, and $\NN=\N\backslash \{0\}$.
When $n_r\geq 2$, the following iterated integral and sum converge absolutely, and we have
\begin{eqnarray}  \label{Iconvdef}
  I_{\gamma}(0;\rho(n_1,\ldots, n_r) ;1)   & = &  (-1)^r \sum_{0<k_1<\ldots< k_r} { 1\over k_1^{n_1} \ldots k_r^{n_r}}  \\
& = & (-1)^r \zeta(n_1,\ldots, n_r) \ .   \nonumber\end{eqnarray}
This is easily  verified from a  geometric expansion of ${dt\over t-1}$.
  In this case, the  word $\rho(n_1,\ldots, n_r) \in \{0,1\}^\times$ begins in $1$
and ends in $0$,   and is called a  convergent word in $0,1$ for obvious reasons.

In general, for any sequence $(n_1,\ldots, n_r) \in \NN^r$, the  quantity $\sum_i n_i$ is called  the weight, and $r$ the depth.

\subsection{Regularization of MZVs}
 One can extend the definition of $I_{\gamma}(0;a_1,\ldots, a_n;1)$ with $a_i\in\{0,1\}$  from the set of convergent words 
to the general case by using  the shuffle product formula. We henceforth drop the $\gamma$ from the subscript.

\begin{lem} There is a unique way to define a set of real numbers $I(a_0;a_1,\ldots,a_n;a_{n+1})$ for any $a_i\in \{0,1\}$, such that

 \begin{itemize} \label{lregdef}

\item $I(0;a_1,\ldots, a_n;1)$ is given by  $(\ref{Iconvdef})$ if $a_1=1$ and $a_n=0$.
\item $I(a_0;a_1;a_2)=0$ and $I(a_0;a_1)=1$ for  all $a_0,a_1,a_2 \in \{0,1\}$. 
\item  (Shuffle product). For all $n=r+s$ and $a_0,\ldots, a_{n+1} \in \{0,1\}$
$$ I (a_0;a_1,\ldots, a_r;a_{n+1}) I(a_0;a_{r+1},\ldots, a_{r+s};a_{n+1})  \quad $$ 
$$ \qquad = \sum_{\sigma \in \Sigma(r,s)} I(a_0;a_{\sigma(1)},\ldots, a_{\sigma(r+s)};a_{n+1})\ .$$
\item $I(a_0;a_1,\ldots, a_n;a_{n+1}) =0   \hbox{ if } a_0=a_{n+1}  \hbox{ and } n\geq1$.
\item $ I(a_0;a_1,\ldots, a_n;a_{n+1}) = (-1)^n I(a_{n+1};a_n,\ldots, a_1;a_0)$.
\item $ I(a_0;a_1,\ldots, a_n;a_{n+1}) =  I(1-a_{n+1};1-a_n,\ldots, 1-a_1;1-a_0)$.
\end{itemize}
\end{lem} 
The second last equation is simply the reversal of paths formula, the last equation is functoriality with respect to the map  $t\mapsto 1-t$.
The numbers $\zeta(n_1,\ldots, n_r)$ defined for any $n_i\in \NN$ by $(-1)^rI(0;\rho(n_1,\ldots, n_r);1)$ are sometimes called  shuffle-regularized multiple zeta values.

\subsection{Structure of MZV's in low weights} \label{sectMZVstructure}

 Let $\mathcal{Z}_N$ denote the $\Q$-vector space spanned by the set of multiple zeta values $\zeta(n_1,\ldots, n_r)$ with $n_r\geq 2$ of total weight $N=n_1+\ldots+n_r$, and let $\mathcal{Z}$ denote the $\Q$-algebra spanned by all multiple zeta values over $\Q$.  It is the sum of the vector spaces $\mathcal{Z}_N\subset \R$, and conjecturally a direct sum.
 By standard lattice reduction methods,  one can  try to write down a conjectural basis for $\mathcal{Z}$ for weight $\leq N$.  Up to weight 10, one experimentally obtains:

\begin{center}
\begin{tabular}{|c|c|c|c|c|c|c|c|c|c|}
  \hline
 Weight  $N$ & 1 & 2 & 3 & 4 & 5    & 6 & 7 & 8 \\
\hline
$\mathcal{Z}_N$& $\emptyset$ & $\zeta(2)$ & $\zeta(3)$ & $\zeta(2)^2$ & $\zeta(5)$ &$\zeta(3)^2$  & $\zeta(7)$ &  $\zeta(3,5)$  \\
&  &   &  &  & $\zeta(3)\zeta(2)$  &  $\zeta(2)^3$ & $\zeta(5)\zeta(2)$  &  $\zeta(3)\zeta(5)$  \\
& & & && &  & $\zeta(3)\zeta(2)^2$  &  $\zeta(3)^2\zeta(2)$  \\
& & & && &  &    &$\zeta(2)^4$  \\
\hline
 $\dim_{\Q}\mathcal{Z}_N$  & 0 & 1 & 1 & 1 & 2   & 2 & 3 & 4   \\
 \hline
\end{tabular}
\end{center}

\begin{center}
\begin{tabular}{|c|c|c|}
  \hline
 Weight  $N$ & 9 & 10  \\
\hline
$\mathcal{Z}_N$& $\zeta(9)$  & $\zeta(3,7)$\\
&  $\zeta(3)^3$ &    $\zeta(3)\zeta(7)$ \\
&  $\zeta(7)\zeta(2)$ & $\zeta(5)^2$  \\
& $\zeta(5) \zeta(2)^2$& $\zeta(3,5)\zeta(2)$   \\
& $\zeta(3) \zeta(2)^3$&  $\zeta(3)\zeta(5)\zeta(2)$\\
& &  $\zeta(3)^2\zeta(2)^2$\\
& &  $\zeta(2)^5$\\
\hline
 $\dim_{\Q}\mathcal{Z}_N$  & 5 & 7 \\
 \hline
\end{tabular}
\end{center}

The dimensions at the bottom are conjectural, and it is not even  known whether $\zeta(5)$ and $\zeta(3)\zeta(2)$ are linearly independent over $\Q$.

For example, the table implies  that there exists a relation between the two multiple zeta values $\zeta(3)$ and $\zeta(1,2)$ in weight 3, and indeed it was shown by Euler that  $\zeta(3)=\zeta(1,2)$.  In  weight 8 there  appears the first multiple zeta value  $\zeta(3,5)$ which  conjecturally cannot be expressed as a polynomial 
in the single zetas $\zeta(n)$ with coefficients in $\Q$. One expects
$$\{\zeta(2), \zeta(3),\zeta(5),\zeta(7),\zeta(3,5),\zeta(9),\zeta(3,7)\}$$
to be algebraically independent over $\Q$.

\section{Motivic formalism} In what follows, all vector spaces etc are defined over the field $\Q$.
\subsection{The category of mixed Tate motives over $\Z$}
Let $\MT(\Z)$ denote the category of mixed Tate motives over $\Z$  \cite{DG}. This is a Tannakian category whose simple objects are the Tate motives $\Q(n)$, indexed by $n\in \Z$, and which have weight $-2n$. The structure of $\MT(\Z)$ is determined  by the data:
\begin{equation}\label{extdim}
\mathrm{Ext}_{\MT(\Z)}^1(\Q(0),\Q(n)) \cong \left\{
                           \begin{array}{ll}
                             \Q\  & \hbox{if } n\geq 3 \hbox{ is odd}\ ,  \\
                             0\   & \hbox{otherwise} \ ,
                           \end{array}
                         \right.
\end{equation}
 and the fact that the $\mathrm{Ext}^2$'s  vanish.  Thus $\MT(\Z)$ is equivalent to the category of representations of a group  scheme $\mathcal{G}_{\MT}$ over $\Q$, which 
 is a semi-direct product
 \begin{equation} \label{Gexact}
 \mathcal{G}_{\MT} \cong \mathcal{G_U} \rtimes \mathbb{G}_m\ ,
 \end{equation}
 where $\Gu$ is the prounipotent algebraic group over $\Q$ whose Lie algebra is the free Lie algebra with one generator  $\sigma_{2n+1}$ in degree $-(2n+1)$. The generators correspond to $(\ref{extdim})$,
and the  freeness follows from the vanishing of the $\mathrm{Ext}^2$'s. The motivic weight is twice the degree.

\begin{rem}\label{remweight}  Henceforth we shall use the word weight  to refer to 
\emph{half} the motivic weight,  in keeping with the usual terminology for MZVs.
\end{rem}
\begin{defn} Let $\Amt$ denote the graded ring of affine functions on $\Gu$  over $\Q$.  
It is a commutative graded Hopf algebra  whose  coproduct we denote by
$$\Delta: \Amt \To \Amt \otimes_{\Q} \Amt\ .$$
Define a trivial comodule over $\Amt$ to be:
\begin{equation}
\Homt = \Amt \otimes_{\Q} \Q[f_2]\ ,
\end{equation} 
where $f_2$ is defined to be of degree 2. As a graded vector space,
 $$\Homt\cong \bigoplus_{k\geq 0} \Amt[2k] \ ,$$
 where $[2k]$ denotes a shift in degree of $+2k$. We also write the coaction:
$$\Delta: \Homt \To \Amt\otimes_{\Q} \Homt\ .$$
It is determined by  its restriction to $\Amt$ and the formula $\Delta(f_2) = 1\otimes f_2$.
\end{defn}

The structure of  $\Homt$ can be described explicitly as follows. It follows from the remarks above  that $\Amt$ is   non-canonically isomorphic to
the cofree  Hopf algebra on cogenerators $f_{2r+1}$ in degree $2r+1\geq 3$: 
$$\U'=\Q\langle f_3,f_5,\ldots \rangle\ .$$
This has a basis consisting  of all non-commutative words in the $f_{\mathrm{odd}}$'s.
The   notation   $\U'$ is superfluous but useful since  we will need to consider many different isomorphisms $\Amt \cong \U'$.
Again, we denote the coproduct on  $\U'$ by $\Delta$, which is  given by deconcatenation:
 \begin{eqnarray} \Delta:  \U' &\To & \U' \otimes_{\Q} \U' \\
 \Delta(f_{i_1} \ldots f_{i_r}) & = &  1\otimes f_{i_1}\ldots f_{i_r} + f_{i_1}\ldots f_{i_r}\otimes 1 \nonumber   \\
 && \qquad \quad + \quad \sum_{k=1}^{r-1} f_{i_1}\ldots f_{i_k} \otimes f_{i_{k+1}} \ldots f_{i_r} \nonumber
 \end{eqnarray} 
The multiplication on $\U'$ is  given by the shuffle product  $(\ref{shuffdef})$.

By analogy with $\Homt$ let us define a trivial comodule
$$ \U= \Q\langle f_3,f_5,\ldots   \rangle \otimes_{\Q} \Q[f_2]$$
where $f_2$  is of degree $2$ and commutes with the $f_{\mathrm{odd}}$.   The coaction 
$$\Delta: \U \To \U' \otimes_\Q \U$$
satisfies $\Delta(f_2) = 1\otimes f_2$.
  The total degree gives a grading $\U_k$ on $\U$ which we call the weight (remark \ref{remweight}).  Thus we  have a non-canonical isomorphism
\begin{equation}  \label{firstpsi}
\psi: \Homt \cong \U \end{equation} 
of graded algebra-comodules, which induces an isomorphism of the underlying graded Hopf algebras $\Amt$ and  $\U'$, and  maps $f_2$ to $f_2$.
 
\begin{lem} Let $d_k=\dim 
\U_k=\dim \Homt_k$. Then
\begin{equation} \label{enumeration} 
\sum_{k\geq 1} d_k t^k = {1\over  1-t^2-t^3}\ . \end{equation}
In particular, $d_0=1, d_1=0, d_2=1$ and  $d_k= d_{k-2}+d_{k-3}$ for $k\geq 3$.
\end{lem} 
\begin{proof} The Poincar\'e series of $\Q\langle f_3,f_5, \ldots \rangle$ is given by 
$${1\over 1-t^3-t^5-\ldots } = {1-t^2 \over 1-t^2-t^3}$$
Multiplying by the Poincar\'e series ${1\over 1-t^2}$  for $\Q[f_2]$ gives $(\ref{enumeration})$.
\end{proof}
 
If we define the depth of $f_{2i+1}$ to be 1 for all $i>0$, and the depth of  $f_2$ to be  0, then 
 we obtain a grading on $\U$ which simply counts the number of odd elements $f_{2i+1}$. 
 The \emph{motivic depth} is the associated  increasing filtration and can be defined in terms of the coaction 
 $\Homt \rightarrow \Amt \otimes_{\Q} \Homt$.
  One checks that  the  motivic depth filtration induced  on $\Homt$ by $(\ref{firstpsi})$
is well-defined, and independent of the choice of $\psi$. In other words, the filtration is motivic, but the grading is not. This stems from the fact that $\sigma_{2i+1}$
is well-defined only up to addition  of commutators of $\sigma_j$ for $j< 2i+1$.

\begin{example} Compare the structure of $\Homt$ in low weights with  the table of multiple zeta values given in \S\ref{sectMZVstructure}:
\vspace{0.1in}

\begin{center}
\begin{tabular}{|c|c|c|c|c|c|c|c|c|c|c|c|}
  \hline
 Weight  $k$ & 1 & 2 & 3 & 4 & 5    & 6 & 7 & 8 & 9 & 10\\
\hline
 & $\emptyset$ & $f_2$ & $f_3$ & $f_2^2$ & $f_5$ &$f_3\!\sha\! f_3$  & $f_7$ &  $f_5f_3$ & $f_9$ & $f_7f_3$  \\
Basis for &  &   &  &  & $f_3f_2$  &  $f_2^3$ & $f_5f_2$  &  $f_3\!\sha\! f_5$  & $f_3\!\sha \!f_3\!\sha \!f_3$ & $f_3\sha f_7$\\
  $\Homt_k$ & & & && &  & $f_3f_2^2$  &  $f_3\!\sha \!f_3 f_2$ & $f_7 f_2$ &  $f_5\sha f_5 $ \\
& & & && &  &    &$f_2^4$  & $f_5f_2^2$ & $f_5f_3f_2$\\
& & & && &  &    &  & $f_3f_2^3$ & $f_3\!\sha \!f_5f_2$\\
& & & && &  &    & &  & $f_3\!\sha\! f_3f_2^2$\\
& & & && &  &    & &  & $f_2^5$\\
\hline
 $\dim$ & 0 & 1 & 1 & 1 & 2   & 2 & 3 & 4 & 5 & 7   \\
 \hline
\end{tabular}
\end{center}
\vspace{0.1in}
\noindent
The following well-known conjecture is of  a transcendental nature.
\begin{conj} The  $\Q$-algebra of MZV's is graded by the weight:  
$$\mathcal{Z} \cong \bigoplus_{k \geq 0} \mathcal{Z}_k$$
 and there is an isomorphism  of graded algebras:
$$\mathcal{Z} \cong \Homt\ .$$
\end{conj}
The first part  implies that there should be no relations between multiple zeta values of different weights.
The second implies in particular that the multiple zeta values should inherit the coaction of the motivic Hopf algebra $\Amt$.  To see what this coaction should be  requires introducing motivic multiple zetas, for which the independence in different weights is automatic.
\end{example}

\subsection{Motivic multiple zeta values.} 
In \cite{GG}, Goncharov showed  how to lift the ordinary iterated integrals $I(a_0;\ldots;a_{n+1})$,  
where $a_i \in \overline{\Q}$ to  periods of mixed Tate motives. In the case where the $a_i \in \{0,1\}$, he showed that these motives are unramified over $\Z$ (see also \cite{GM}), 
and therefore define  objects in $\Amt$. In his version of motivic multiple zeta values, the element corresponding to $\zeta(2)$ is zero.

One can show  using the formalism of \cite{DG} that these can in turn be lifted to elements of $\Homt$ in such a way that the motivic version of $\zeta(2)$ is non-zero. However, the tollation involves making some choices (see  \cite{Br23}, \S2 for the definitions). In summary:

\begin{thm}  \label{propdefI}  There exists a sub-Hopf algebra $\Ao\subset \Amt$ and  a  graded algebra-comodule $\Ho$ over $\Ao$, which satisfies the following properties.
It is spanned by  elements (called motivic iterated integrals)
\begin{equation}\label{Iframed} 
\Imot(a_0;a_1,\ldots, a_n;a_{n+1}) \in \Ho_{n}
\end{equation}
where  $a_0,\ldots, a_{n+1} \in \{0,1\}$,  such that:
\begin{description} \label{lemImotrelations}
\item[I0] $\Imot(a_0;a_1,\ldots, a_n;a_{n+1}) =0  \hbox{ if } a_0=a_{n+1}  \hbox{ and } n\geq1$
\item[I1] $\Imot(a_0;a_1;a_2)=0$ and $\Imot(a_0;a_1) = 1$ for all $a_0,a_1,a_2 \in \{0,1\}$
\item[I2] $ \Imot(0;a_1,\ldots, a_n;1) = (-1)^n \Imot(1;a_n,\ldots, a_1;0) $
\item[I3] $\Imot(0;a_1,\ldots, a_n;1) =   \Imot(0;1-a_n,\ldots, 1-a_1;1)$
\end{description}
Furthermore, for any $a_i,x,y \in \{0,1\}$ the shuffle product formula holds:
 $$\Imot(x;a_1,\ldots, a_r;y) \Imot(x;a_{r+1},\ldots, a_{r+s};y) = \sum_{\sigma \in \Sigma(r,s)} \Imot(x;a_{\sigma(1)},\ldots, a_{\sigma(r+s)};y)\ .$$
 There is a well-defined map (the  period)
 \begin{eqnarray} \label{periodmap} per: \Ho &\rightarrow&  \R   \\
 \Imot(a_0;a_1,\ldots, a_n;a_{n+1}) & \To & I(a_0;a_1,\ldots,a_n; a_{n+1}) \nonumber 
   \end{eqnarray}
   which is a ring homomorphism. 
   In particular, all relations satisfied by the $\Imot(a_0;a_1,\ldots, a_n;a_{n+1})$ are also satisfied by the $I(a_0;a_1,\ldots, a_n;a_{n+1})$.
 
 Finally, there is a non-canonical isomorphism
\begin{equation}\label{Hotens} \Ho \cong \Ao\otimes_{\Q} \Q[\zetam(2)] \ ,\end{equation}
where $\zetam(2)$ denotes the motivic iterated integral $-\Imot(0;1,0;1)$. As a consequence, there is a non-canonical embedding of algebra-comodules 
\begin{equation}\label{Hoembed}
\Ho \hookrightarrow \Homt 
\end{equation}\
which maps $\zetam(2)$ to $f_2$.
 \end{thm}
\begin{defn} Let $n_1,\ldots, n_r\in \NN$, where $n_r\geq 2$. Define the \emph{motivic multiple zeta value} to be the element in $\Ho$ given by:
$$\zetam(n_1,\ldots, n_r) = (-1)^n \Imot(0; \rho(n_1,\ldots, n_r) ;1)\ .$$
Its period is $\zeta(n_1,\ldots, n_r)$.
\end{defn}
Note that in our setting the element $\zetam(2)$ is non-zero.  

 \begin{rem} 
 The preceding theorem is rather powerful. For instance, it immediately implies that 
 $$\dim_{\Q} \mathcal{Z}_k \leq \dim_{\Q} \Ho_{k} \leq \dim_{\Q} \Homt _k= d_k$$
 where the numbers  $d_k$ are defined by $(\ref{enumeration})$. This theorem was first proved independently  by Goncharov (see Deligne-Goncharov \cite{DG}) and  Terasoma \cite{T}. This  upper bound on    $ \dim_{\Q} \Ho_{k}$ comes from $(\ref{Hoembed})$. 
 The main result of \cite{Br23} is  the lower bound    $ \dim_{\Q} \Ho_{k}\geq d_k$, which  in turn implies that  $(\ref{Hoembed})$ is an isomorphism. We shall not need this fact for the sequel. 
  \end{rem}

The various choices made above will be absorbed into a single morphism  
 of graded algebra-comodules \begin{equation}  \label{firstphi}
\phi: \Ho \To  \U\end{equation}
which is obtained by composing $(\ref{Hoembed})$ with $(\ref{firstpsi})$.
 It maps $\zetam(2)$ to $f_2$, and induces a morphism of Hopf algebras $\phi: \Ao \rightarrow \U'$.

 \subsection{Notations} The motivic multiple zeta values can exist on three different levels: the highest being  the comodule $\Ho$; next 
 the Hopf algebra 
 $$\Ao = \Ho / \zetam(2) \Ho$$
 in which $\zetam(2)$ is killed; and finally the Lie coalgebra 
 \begin{equation}\label{lodef}
\Lo = {\Ao_{>0} \over \Ao_{>0} \Ao_{>0}}\ ,
\end{equation}
of indecomposable elements of $\Ao$. We use the notation $\zetam$ to denote an element in $\Ho$; $\zetaa$ its image in $\Ao$; and $\zetal$ its image in $\Lo$: 
  \begin{equation}
 \begin{array}{ccccc}
 \Ho_{>0}  & \To  &   \Ao_{>0} &  \To & \Lo \\
 \rin &   &   \rin &   & \rin \\
 \zetam(w) & \mapsto  & \zetaa(w)  & \mapsto & \zetal(w)
\end{array}
 \end{equation}
Thus the elements $\zetaa(n_1,\ldots, n_r)$ are exactly the motivic multiple zeta values considered by Goncharov in \cite{GG}, and $\zetaa(2)=0$.
 We use the same superscripts for  the motivic iterated integrals, viz. $\Imot$, $\Imota$, $\Imotl$.
 
\subsection{Formula for the coaction} 
Goncharov computed the coproduct $\Delta: \Ao \rightarrow \Ao \otimes_{\Q} \Ao$  on the elements $\Imota(a_0;\ldots; a_{n+1})$ in \cite{GG}, Theorem 1.2.
The coaction on $\Ho$ is given by the same formula, after interchanging the two right-hand factors (see \cite{Br23}, \S2).
\begin{thm} \label{thmGonchCoproduct}  The coaction
 \begin{equation} \label{Hcoaction}
 \Delta: \Ho \To \Ao \otimes_{\Q} \Ho\ ,
 \end{equation} 
 can be computed explicitly as follows. 
For any $a_0,\ldots, a_{n+1}\in \{0,1\}$,      the  image of a generator   $\Delta\,  \Imot(a_0;a_1,\ldots, a_{n};a_{n+1})$ is given by 
\begin{equation}  \sum_{i_0<i_1<  \ldots<  i_k<i_{k+1}}  
 \!\!\! \Big( \prod_{p=0}^k \Imota(a_{i_p}; a_{i_p+1}, \tdots ,a_{i_{p+1}-1} ;a_{i_{p+1}}) \Big)\otimes \Imot(a_0;a_{i_1},\tdots, a_{i_k}; a_{n+1})  \nonumber \end{equation}
where  the sum is over indices satisfying $i_0=0$ and $i_{k+1}=n+1$, and all   $0\leq k\leq n$. Note that the trivial elements  $\Imota(a;b)$ are equal to $1$.
\end{thm}
This formula has an elegant   interpretation in terms of cutting off segments of a semicircular polygon, for which we refer to \cite{GG} for further details.

\subsection{Zeta cogenerators.} 
The following lemma (\cite{GG}, Theorem 6.4) is an easy consequence of 
 theorem  \ref{propdefI}, theorem \ref{thmGonchCoproduct}, and  the fact that  $\zeta(2n+1) \neq 0$. 
\begin{lem} For $n\geq 1$,   $\zetam(2n+1)\in \Ho$ is non-zero and satisfies
$$\Delta\, \zetam(2n+1) = 1 \otimes \zetam(2n+1) + \zetaa(2n+1) \otimes 1\ .$$
Furthermore, Euler's relation for even zeta values implies that
$$\zetam(2n) = b_n \zetam(2)^n $$
where 
 $b_n=(-1)^{n+1}{1\over 2} B_{2n} {(24)^n \over (2n)!}$, and the $B_{2n}$ are Bernoulli numbers.
 \end{lem}

We can therefore normalize our choice of  map   $(\ref{firstphi})$  so that
$$\Ho \overset{\phi}{\To} \U  $$
maps   $\zetam(2n+1)$ to $f_{2n+1}$.   For notational convenience  we define
\begin{equation} \label{fndef} 
f_{2n} = b_n\, f_2^n   \in \U_{2n} 
\end{equation}
where $b_n$ is defined in the previous lemma. We can therefore write:
\begin{equation} \label{phinormed}
\phi(\zetam(N))=f_N \quad   \hbox{ for all } \quad  N\geq 2 \ . 
\end{equation}

\begin{rem}
If $\xi \in \Ho$ is  of weight $N$ then $\xi'=\xi+\alpha\, \zetam(N)$, for any $\alpha \in \Q$,  cannot be distinguished from $\xi$ using the coaction $\Delta$. This is the basic
reason why our decomposition algorithm ($\S5$) is not exact.
\end{rem}

\section{The derivations $\partial_{2n+1}$} In order to simplify the formula for the coaction  $(\ref{Hcoaction})$, it is convenient to consider an infinitesimal version of it. We first consider the comodule $\U$.

\subsection{Truncation operators on $\U$}  In order  to detect elements in $\U$ we can use a set of derivations as follows. 
For each $n\geq 1$, define truncation maps 
\begin{eqnarray} \label{truncdef}
\partial_{2n+1}: \Q\langle f_3,f_5,\ldots \rangle  & \rightarrow & \Q\langle f_3,f_5,\ldots  \rangle \\
 \partial_{2n+1} (f_{i_1}\ldots f_{i_r}) & =  &\left\{
                           \begin{array}{ll}
                             f_{i_2}\ldots f_{i_r} , & \hbox{if } \quad i_1=2n+1\ ,  \nonumber \\
                             0 , & \hbox{otherwise} \ .
                           \end{array}
                         \right.
 \end{eqnarray}
It is easy to verify that $\partial_{2n+1}$ is a derivation for the shuffle product, i.e., 
$$\partial_{2n+1} (a\sha b) = \partial_{2n+1}(a) \sha b + a \sha \partial_{2n+1}(b) \ ,$$
for any $a,b\in \Q\langle f_3,f_5,\ldots  \rangle$.
The map $\partial_{2n+1}$  decreases the motivic depth by 1, and the weight by $2n+1$.
If we set $\partial_{2n+1} (f_2)=0$, then the maps $\partial_{2n+1}$ uniquely extend to derivations:
$$\partial_{2n+1} : \U \To \U \ . $$
\begin{defn} Let $\partial_{<N}$ be the  sum of   $\partial_{2i+1}$ for $1<2i+1<N$:
\begin{equation}\label{deltaN} \partial_{<N} : \U_N \To \bigoplus_{1\leq i<\lfloor {N\over 2} \rfloor} \U_{N-2i-1}
\end{equation}
\end{defn}

\begin{lem} The following sequence is  exact: 
\begin{equation} \label{keronedim}   0 \To f_N \Q \To \U_N \overset{\partial_{<N}}{\To} \bigoplus_{1\leq i< \lfloor {N\over 2} \rfloor} \U_{N-2i-1} \To 0
\end{equation}
\end{lem}
\begin{proof} It is clear that every element $F \in \U_N$ can be uniquely written:
\begin{equation}\label{xiexpand}
F= \sum_{1\leq i < \lfloor  {N\over 2} \rfloor} f_{2i+1} v_{N-2i-1} + c f_N \end{equation}
where $c\in \Q$ and the  $v_{j} \in \U_{j}$. The elements $v_{N-2i-1}$ are equal to $\partial_{2i+1} F$ by definition. Every tuple $(v_{N-2i-1})_{1\leq i < \lfloor {N\over 2} \rfloor} $  
arises in this way. \end{proof}
Thus by repeatedly applying operators $\partial_{2i+1}$ for $2i+1<N$, we can   detect elements in $\U_N$, up to elements in  the kernel $f_N \Q$. 

\subsection{Hopf algebra interpretation}  Recalling that $\U'= \U/f_2$,   consider the set   of indecomposables:
$$L = { \U'_{>0} \over \U'_{>0} \U'_{>0}}\ ,$$
which is the cofree Lie coalgebra on cogenerators $f_3,f_5,\ldots$ in all odd degrees $\geq 3$. 
Its  (weight) graded  dual  $L^\vee$ is the free Lie algebra on dual generators $f^\vee_3,f^\vee_5,\ldots $ in all negative odd degrees $\leq-3$. 
In each graded weight $N$  there is a perfect pairing $L_N\otimes_{\Q}L_N^\vee \rightarrow \Q$ of finite-dimensional vector spaces. 
Thus every dual generator defines a map $f^{\vee}_{2n+1}: L \rightarrow \Q$.
Let 
$\pi: \U'_{>0} \rightarrow L$ denote the  quotient map, and for $2n+1\leq N$ consider the map
\begin{equation}  \label{Hopfdelta}
\U\overset{\Delta'}{\To} \U'_{>0} \otimes_{\Q} \U \overset{\pi \otimes id}{\To} L\otimes_{\Q} \U \overset{f_{2n+1}^{\vee}\otimes id}{\To}  \U
\end{equation}
where $\Delta' = \Delta - 1 \otimes id$.
It follows from the structure of $\U$ that this map is precisely $\partial_{2n+1}$ $(\ref{truncdef})$. 
   Note that    $(\ref{Hopfdelta})$, restricted to $\U_N$,  factors  through:
\begin{equation} \label{factoredDelta} \U_N \To   \U'_{2n+1}\otimes_{\Q} \U_{N-2n-1} \overset{\pi\otimes id}{\To}  L_{2n+1} \otimes_{\Q} U_{N-2n-1}
\end{equation} 
where the first map is the $(2n+1,N-2n-1)$-graded part of $\Delta$. 
 
\subsection{Derivations on $\Ho$}  The previous constructions can be transferred to the Hopf algebra $\Ho$. 
First observe that $\Ho_{\leq N }\subset \Ho$ and $\U_{\leq N} \subset \U$ are subcoalgebras.
 Suppose  that we have a linear bijection  up to  weight $N$:
\begin{equation} \label{weightiso}
\phi: \Ho_{\leq  N}  \overset{\sim}{\To} U_{\leq N}
\end{equation} 
which respects the comodule structures, i.e.,  $\Delta \phi=\phi \Delta$,  and also the multiplication laws, i.e., 
$\phi(x_1 x_2) = \phi(x_1)\phi(x_2)$ for all $x_1,x_2\in \Ho$ such that $\deg x_1+\deg x_2 \leq N$.
Then every element of $\Ho_{\leq N}$, and in particular every motivic multiple zeta value of weight  less than or equal to $N$,  can be identified with a non-commutative polynomial in the generators $f_{i}$.\footnote{We know by \cite{Br23} that such a $\phi$ exists for all $N$.}  

Transporting via the map $\phi$ leads to derivations
$$\partial^{\phi}_{2n+1} =\phi^{-1} \circ  \partial_{2n+1} \circ \phi   $$
for all $2n+1\leq N$.  These define derivations on the whole of $\Ho$, but for the purposes of the present paper we shall only need to consider their restriction 
 $\partial^{\phi}_{2n+1}:\Ho_{\leq N} \rightarrow \Ho_{\leq N-2n-1}$ .
By analogy with $\partial_{<N}$,  we define 
\begin{equation}\label{deltaphilessNdefn}
\partial^{\phi}_{<N} = \bigoplus_{1\leq i < \lfloor { N \over 2} \rfloor}  \partial^{\phi}_{2i+1} \ . 
\end{equation}
   We shall compute the derivations $\partial^{\phi}_{2i+1}$ in the following way.
Let 
$$\pi : \Ao_{>0} \rightarrow \Lo$$
 denote the quotient map, where $\Lo$ is the Lie coalgebra of indecomposables $(\ref{lodef})$.  We denote the map  $ \Lo_{\leq N} \rightarrow L_{\leq N}$ induced by $(\ref{weightiso})$ by $\phi$ also.
\begin{defn} For all $2n+1\leq N$, define the \emph{coefficient map} to be 
 $$c^{\phi}_{2n+1}=f_{2n+1}^\vee\circ \phi:\Lo_{2n+1} \To \Q\ . $$
\end{defn} 
We shall sometimes extend the coefficient map to $\Ao_{2n+1}$ and $\Ho_{2n+1}$, and denote it by $c^{\phi}_{2n+1}$ also. For an element $\xi \in \Ho_{2n+1}$, the number $c^{\phi}_{2n+1}(\xi)$
is simply  the coefficient of $f_{2n+1}$ in the expansion  $(\ref{xiexpand})$  of $\phi(\xi)$ as a non-commutative  polynomial in the $f$'s.
\begin{defn} \label{Drdef}
For each  odd $r\geq 3$, define 
$$D_r: \Ho_{N} \overset{\Delta_{r,N-r}}{\To} \Ao_r \otimes_{\Q} \Ho_{N-r} \overset{\pi\otimes id}{\To} \Lo_{r}\otimes_{\Q} \Ho_{N-r}$$ 
to be the weight $(r,N-r)$-graded part of the coaction, followed by projection onto the Lie coalgebra.
It follows  from theorem  \ref{thmGonchCoproduct}  that the action of $D_r$  on the element $ \Imot(a_0;a_1,\ldots, a_n;a_{n+1})$ is given explicitly   by:
\begin{equation}\label{mainformula}   
 \sum_{p=0}^{n-r} \Imotl(a_{p} ;a_{p+1},\tdots, a_{p+r}; a_{p+r+1}) \otimes   \Imot(a_{0}; a_1, \tdots, a_{p}, a_{p+r+1}, \tdots ,  a_n ;a_{n+1}) \ .
 \nonumber \end{equation}
Note that this formula is closely related to the Connes-Kreimer  coproduct formula  for  a class of linear graphs  with two external legs.
 By analogy,  we call the sequence $(a_p;a_{p+1},\ldots, a_{p+r};a_{p+r+1})$ on the left the \emph{subsequence} and the sequence 
 $(a_{0}; a_1, \tdots, a_{p}, a_{p+r+1}, \tdots ,  a_n ;a_{n+1})$ on the right the \emph{quotient sequence} of our original sequence $(a_0;a_1,\ldots, a_n;a_{n+1})$.
 \end{defn}

It follows from the above that 
\begin{equation}\label{explicitphidelta} \partial^{\phi}_{2n+1} =(c^{\phi}_{2n+1} \otimes id)\circ D_{2n+1}\ .
\end{equation}
Only the coefficient map depends on the choice of $\phi$.

\subsection{Normalization of $\phi$ in depth 1} \label{sectNormalized} In order to put  the operators $\partial^{\phi}_{2n+1}$ to use we first have to choose an  isomorphism $\phi$.
We shall always assume that $\phi$ is normalized so that 
$$\phi(\zetam(2n+1)) = f_{2n+1}$$
for all $2n+1\leq N$.  The coefficient $c^{\phi}_{2n+1} \zetam(2n+1)$  is  therefore 1.
 By the shuffle relations for motivic iterated integrals, one can check that 
 \begin{equation}\label{Idepth1shuff}
 \Imot(0; \underbrace{0,\ldots, 0}_a, 1, \underbrace{0,\ldots, 0}_{2n-a};1)  = (-1)^{a} \binom{2n}{a} \zetam(2n+1)\ .
  \end{equation}
Therefore for any normalized $\phi$  we have
  \begin{equation} \label{normphiproj}
c^{\phi}_{2n+1}  \, (\Imotl(0; \underbrace{0,\ldots, 0}_a, 1, \underbrace{0,\ldots, 0}_{2n-a};1))  = (-1)^{a} \binom{2n}{a}\ .
\end{equation}
In the later examples, this equation will be  used many times.

\begin{ex} \label{ex2343} 
We compute the operators $D_r$ on some examples.
\vspace{0.05in}

\emph{i).} Consider the element  $\zetam(2,3) =\Imot(0;10100;1) \in \Ho_5$. We have
$${\small D_3 \Imot(0;10100;1) =   \Imotl(1;010;0) \otimes \Imot(0;10;1) +\Imotl(0;100;1)\otimes  \Imot(0;10;1) } $$
 The   reflection  relation  yields  $\Imot(1;010;0) = - \Imot(0;010;1)$  which equals $ 2\Imot(0;100;1)$ by $(\ref{Idepth1shuff})$,  so we conclude that
$D_3 \zetam(2,3) = 3\,  \zetal(3) \otimes \zetam(2)$. In particular for any normalized $\phi$, we have
$\partial^{\phi}_3 \zetam(2,3) = 3\, \zetam(2)\ .$ Thus
$\phi(\zetam(2,3))=3 f_3f_2 + cf_5$ where $c\in \Q$ remains to be determined.
\vspace{0.05in}

\quad  \emph{ii).} Consider   $\zetam(4,3)=\Imot(0;1000100;1) \in \Ho_7$. From  $(\ref{mainformula})$ ,
{\small \begin{eqnarray}
D_3 \Imot(0;1000100;1) & =  & \Imotl(0;100;1) \otimes \Imot(0;1000;1)    \nonumber \\
 & = & \zetal(3)
 \otimes \zetam(4) \nonumber \\
D_5 \Imot(0;1000100;1) & = &   \Imotl(1;00010;0)  \otimes  \Imot(0;10;1)+   \Imotl(0;00100;1)\otimes  \Imot(0;10;1)  \nonumber   \\
 & =  &  10\,  \zetal(5) \otimes \zetam(2)    \nonumber 
 \end{eqnarray}}
 \noindent Thus, for a normalized $\phi$,   $\partial^{\phi}_3 \zetam(4,3) = \zetam(4)$ and $\partial^{\phi}_5 \zetam(4,3) =10\,  \zetam(2)$.
 Hence $\phi(\zetam(4,3))= f_3 f_4+ 10 f_5 f_2 + c f_7$, where $c\in \Q$ is to be calculated.
\end{ex}
These examples can be  depicted graphically  as follows. The derivations  above  cut off a segment from the marked semi-circles
indicated below. Only the segments which give non-zero contributions are indicated.
\begin{figure}[h!]
 \begin{center}
    \leavevmode
    \epsfxsize=10.0cm \epsfbox{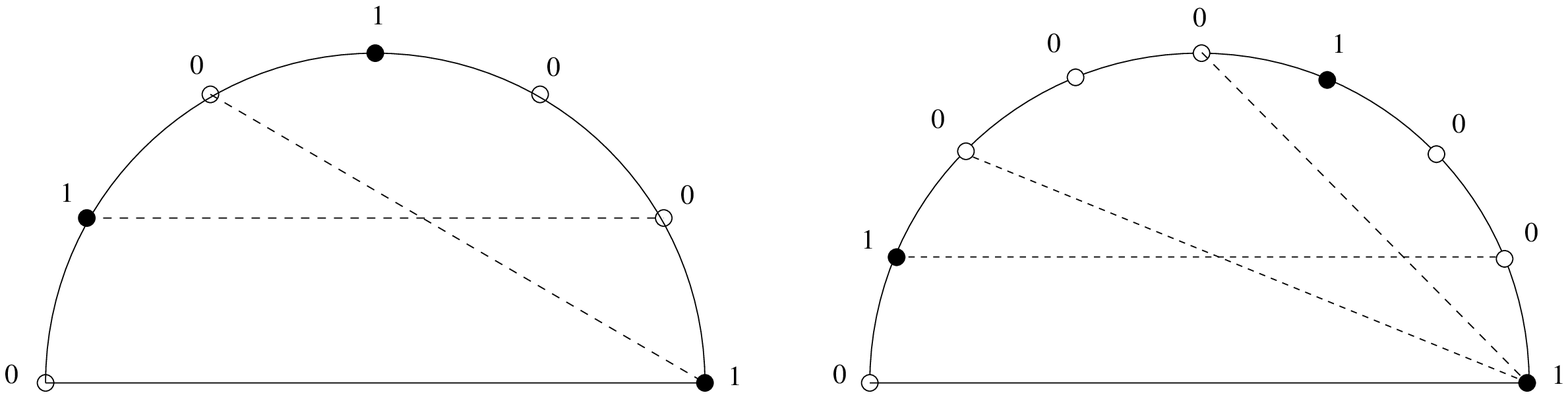}
 \label{2gen}
 \put(-250,-20){$\Imot(0;10100;1)$}  \put(-100,-20){$\Imot(0;1000100;1)$}
  \end{center}
\end{figure}

It follows from $(\ref{keronedim})$ that  the operators $D_{2r+1}$ yield a lot  of explicit information about  multiple zeta values and their motivic versions.

 As a further illustration, consider the family of elements
 $$\zetam(1,3,\ldots, 1,3)=\Imot(0;1100\ldots 1100;1)$$
 Any subsequence of odd length $2r+1$ of  $1100\ldots 1100$ necessarily begins and ends with the same symbol, and so the corresponding motivic
 iterated integral vanishes, by {\bf I0}. It follows that for any $\phi$, $\partial^{\phi}_{2r+1} \zetam(3,1,\ldots, 3,1)=0$ for all $r\geq 1$. 
 Therefore by $(\ref{keronedim})$  the element $\zetam(1,3,\ldots, 1,3)$ is a rational multiple of $\zetam(N)$, where $N$ is its weight. 
 On taking the period map we deduce that
 $$\zeta(\underbrace{1,3,\ldots, 1,3}_n) = \alpha_n \pi^{4n} $$
 for some $\alpha_n\in \Q$. David Broadhurst showed that $\alpha_n = {1 \over (2n+1)(4n+1)!}$.
 
\section{A decomposition algorithm} \label{sectdecompalg} 
By using the comodule structure of  $\U$ and the explicit formula for the operators $D_{2r+1}$, one obtains an `exact-numerical' algorithm for the 
decomposition of multiple zeta values into any  predefined (algebra) basis.

\subsection{Preliminary definitions}

Suppose that we wish to decompose multiple zeta values up to some weight $M\geq 2$. We need the following set-up.
\2skip

 {\bf 1)}.  For  $2\leq N\leq M$ let $V_N$ be the $\Q$-vector space  spanned by symbols:
 \begin{equation} \label{Alzetagen}
 \zetam(n_1,\ldots, n_r) \end{equation}
 where $n_i\geq 1$, $n_r \geq 2$, and $n_1+\ldots + n_r = N$. We call  $N$ the weight.  
 We  also represent  these  elements another way using  a different set of symbols
 \begin{equation} \label{AlImotgen} \Imot(a_0;a_1,\ldots, a_N; a_{N+1})\qquad \hbox{where } a_i\in \{0,1\}\ .\end{equation} 
Any symbol $(\ref{AlImotgen})$ can be reduced to a linear combination of elements of the form $(\ref{Alzetagen})$ using the following relations:
 \begin{description}

\item[R0]   For $n_i\geq 1$, $n_r \geq 2$, and $n_1+\ldots + n_r = N$, we set
$$  \Imot(0;\underbrace{1,0,\ldots, 0}_{n_1},\ldots, \underbrace{1,0,\ldots, 0}_{n_r} ;1)=  (-1)^r\zetam(n_1,\ldots, n_r)\in V_N$$ 
  \item[R1]    $\Imot(a_0;a_1,\ldots,a_N; a_{N+1})=0 \hbox{ if } a_0 =a_{N+1}$  or $a_1=\ldots =a_N$\ .  \2skip
   \item[R2] For $k,  n_1,\ldots, n_r\geq 1$, 
  $$ (-1)^k  \Imot(0;\underbrace{0,\ldots,0}_k, \underbrace{1,0,\ldots, 0}_{n_1},\ldots, \underbrace{1,0,\ldots, 0}_{n_r} ;1)=  $$
 $$  \sum_{i_1+ \ldots +i_r=k} \! \binom{n_1+i_1-1}{i_1}\ldots \binom{n_r+i_r-1}{i_r}  
 \Imot(0;\underbrace{1,0,\ldots, 0}_{n_1+i_1},\ldots, \underbrace{1,0,\ldots, 0}_{n_r+i_r} ;1) $$ \2skip 
  \item[R3] $\Imot(0;a_1,\ldots, a_N;1) = (-1)^n \Imot(1;a_N,\ldots, a_1;0) $ \2skip
  \item[R4] $\Imot(0;a_1,\ldots, a_N;1) =  \Imot(0;1-a_N,\ldots, 1-a_1;1)$ \2skip
\end{description}
To see this, take any element of the form $(\ref{AlImotgen})$ and  use 
 {\bf R1} and {\bf R3} to ensure that $a_0=0$ and $a_{N+1}=1$. Then use {\bf R2} to rewrite it as a linear combination of elements satisfying   $a_1=1$. By {\bf R4} this ensures that $a_N=0$ and finally apply {\bf R2}
 once more to force $a_1=1$.  Conclude using  {\bf R0}.
 
 \begin{rem} Relations  {\bf R0} and  {\bf R4} actually induce an extra relation (known as duality) on the generators $(\ref{Alzetagen})$. One could  take the quotient of $V_N$ modulo this relation if  one chooses, but we shall not do this here.
 \end{rem}
 Finally, for any generator of $V_N$, define its period to be the real number
 \begin{equation} \label{Alper} per( \zetam(n_1,\ldots, n_r)  ) = 
 \zeta(n_1,\ldots, n_r)\in \R\ .\end{equation}
\2skip

{\bf 2)}. For  $2 \leq N\leq M$ define a $\Q$-vector space $\U_N$ with basis elements
\begin{equation} \label{Alfgen}
 f_{2i_1+1} \ldots f_{2i_r+1} f_2^k 
 \end{equation}
where $r,k \geq 0, i_1,\ldots,i_r\geq 1$,  and $2(i_1+\ldots+i_r)+r+2k=N$.  We also need the multiplication rule
$\sha: \U_m\times \U_n \rightarrow \U_{m+n}$ defined by 
$$  f_{2i_1+1} \ldots f_{2i_r+1} f_2^k \,\sha \, f_{2i_{r+1}+1} \ldots f_{2i_{r+s}+1} f_2^\ell $$
$$ \qquad \qquad = \sum_{\sigma \in \Sigma(r,s)}  f_{2i_{\sigma(1)}+1} \ldots f_{2i_{\sigma(r+s)}+1} f_2^{k+\ell} $$
where $  \Sigma(r,s)$ is the set  of $(r,s)$ shuffles, i.e.,  permutations $\sigma$ of $1,\ldots, r+s$ such that $\sigma(1)< \ldots < \sigma(r) $ 
and  $\sigma(r+1)< \ldots < \sigma(r+s). $ 
\2skip 

{\bf 3)}.  Suppose that we have some conjectural polynomial basis of  (motivic) multiple zeta values $B\subset \bigoplus_{2\leq n \leq M} V_n$  up to weight $M$.  We shall assume that $B$ contains the elements
$$B^0=\{ \zetam(2)\} \cup \{\zetam(3),\zetam(5), \ldots, \zetam(2r+1) \}$$
where $r$ is the largest integer such that $2r+1\leq M$. Denote the remaining elements of $B$ by  
$B'=B \backslash B^0,$ and let $B_n$ denote the set of elements of $B$ of weight $n$.
For $2\leq N\leq M$, let $\langle B\rangle_N$ denote  the  $\Q$-vector space spanned by  monomials in elements of the set $B$
which are of total weight $N$, where the weight is additive with respect to multiplication.    Part of the decomposition algorithm is to verify that $B$ is indeed a polynomial basis for the (motivic) multiple zeta values.
 As a first check, one should have 
\begin{equation} \label{Bnormassumption}     
\dim_{\Q} \langle B\rangle_N =d_N  \hbox{ for all } 2\leq N\leq M\ ,
\end{equation}
where  $d_0=1, d_1=0, d_2=1$ and  $d_k= d_{k-2}+d_{k-3}$ for $k\geq 3$.  The integer $d_N$ is the dimension of the vector space $\U_N$.
\2skip

\subsection{Inductive definition of the algorithm}
The algorithm is defined by induction on the weight  and has two parts:

\begin{enumerate}
 \item For all $n\leq N$, we construct a map $$\phi: B_n \rightarrow \U_n\ ,$$ which assigns a  $\Q$-linear combination of monomials  of the form $(\ref{Alfgen})$ to every element of our basis $B$
  of weight at most $N$.  Using the multiplication law $\sha$,  extend this map multiplicatively  to monomials in  the elements of $B$ to give a map
  $$\rho : \langle B\rangle_n \To \U_n$$
  for all $n\leq N$.  We require that $\rho$ be an isomorphism to continue (otherwise, the present  choice  $B$ is not a basis).

\item An algorithm to extend  $\phi$  to the whole of $V_n$:
\begin{equation}\label{AlphinVndef}
\phi: V_n \To \U_n
\end{equation}
 for all $n\leq N$. Thus  there is an \emph{algorithm} to assign  a  $\Q$-linear combination of monomials  of the form $(\ref{Alfgen})$ to every element $(\ref{Alzetagen})$, but note that it does not actually need to be computed explicitly on all elements of $V_n$, only on the basis elements $B_n$.
\end{enumerate}

Once $(1)$ and $(2$)  have been constructed, they give a way to decompose any element $\xi\in V_N$ as a polynomial in our basis: simply compute 
$$\rho^{-1} (\phi(\xi)) \in \langle B\rangle_N\ .$$
\2skip
We now show how to define $(1)$ and $(2)$ by a bootstrapping procedure. 
Suppose that they have been constructed up to and including weight $N$.\footnote{For the intial  case $N=2$, simply set $\phi(\zetam(2))=f_2$.}

From $(2)$, we  have an algorithm to compute   a set of  coefficient functions
\begin{equation} \label{Alcoeffns}
c^{\phi}_{2r+1} : V_{2r+1} \To \Q
\end{equation}
for all $2r+1\leq N$, which to any element $\xi\in V_{2r+1} $ takes the coefficient of the monomial $f_{2r+1}$ in $\phi(\xi)\in \U_{2r+1}$. The induction steps are:
\2skip

{\bf Step 1}.   Define $\phi$ on elements $\xi \in B_{N+1}$ as follows. If $\xi =\zetam(2n+1) $ then   set
$\phi(\xi)= f_{2n+1}$.  Otherwise,   write $\xi$ (or $-\xi$)  in the form 
\begin{equation}\label{Alxiform}
\xi = \Imot(a_0;a_1,\ldots, a_{N+1}; a_{N+2})\end{equation}
 where $a_i \in \{0,1\}$, using relation {\bf R0}. Define for all $3\leq 2r+1 \leq N$, 
\begin{equation}\label{Alxi2r} 
 \xi_{2r+1} = \sum_{p=0}^{N+1-2r} c^{\phi}_{2r+1} \big( \Imot( a_p;a_{p+1},\ldots, a_{p+2r+1}; a_{p+2r+2})\big) \times \end{equation}
 $$\qquad \qquad   \qquad   \qquad  \qquad   \phi( \Imot( a_0;a_{1},\ldots,a_{p},  a_{p+2r+2},\ldots, a_{N+1}; a_{N+2})) $$ 
Then $\xi_{2r+1} \in \U_{2r+1}$ ($\xi_{2r+1}$ is denoted $\partial_{2r+1} \xi$ in the examples in \S6). The right hand  side of the product is computed   using the algorithm 
for $\phi$ in strictly lower weights $(\ref{AlphinVndef})$.  Finally, define
$$\phi(\xi) = \sum_{3\leq 2r+1 \leq N} f_{2r+1} \xi_{2r+1}\ ,$$
where the product on the right is  concatenation. Having computed $\phi$ explicitly on the elements of $B_{N+1}$, compute the map $\rho:\langle B\rangle_{N+1} \rightarrow \U_{N+1}$  by extending $\phi$ by multiplicativity,  and check that it is an isomorphism. If not, then the choice of $B$ is not a basis. In the case when  $B$ contains  linear combinations of terms of the form $( \ref{Alxiform})$, $\phi$ is extended by  linearity and computed in exactly the same way. 
\2skip

{\bf Step 2}. The algorithm to compute $\phi$ on any generator $\xi \in V_{N+1}$ proceeds as follows. As above, write $\xi$ in the form $(\ref{Alxiform})$, and compute
$\xi_{2r+1}$ for $3\leq 2r+1 \leq N$ using the formula  $(\ref{Alxi2r})$.  As before, let
$$u = \sum_{3\leq 2r+1 \leq N} f_{2r+1} \xi_{2r+1}\ .$$
Then $u$ is an element of $\U_{N+1}$, and we can compute $\rho^{-1}( u) \in \langle B\rangle_{N+1}$ as a polynomial in our basis $B$. 
The general theory tells us that
\begin{equation}Ê\label{cncoeff}
c_{\xi} = {per(\xi - \rho^{-1}( u) ) \over \zeta(N+1)} \in \R
\end{equation}
is a rational number. Compute it to as many digits as required in order to identify this rational to a satisfactory degree of certainty. Define
$$
\phi(\xi) = u + c_{\xi} f_{N+1}\ ,
$$
where $f_{2n} = {\zeta(2n)\over \zeta(2)^n} f_2^n$ in the case where $N+1=2n$ is even.
\2skip

Some worked examples  of this algorithm are computed in $\S6$.

\subsection{Comments}
\emph{i).}  In order to  decompose  an element  $\zetam(n_1,\ldots, n_r)$  of weight $N$ into the basis, one must  also decompose  all the sub and quotient sequences of $\Imot(0;\rho(n_1,\ldots, n_r);1)$
as they occur in the definition of $D_{2r+1}$.  Since such sequences have strictly smaller  weight and smaller numbers of $1$'s, the total number of elements to decompose is 
under control.
\vspace{0.05in}

\emph{ii).}  The computation of the coefficients $(\ref{cncoeff})$ requires an efficient numerical method for computing the multiple zeta values.  There are many ways to do this. A simplistic way is to write the path from $0$ to $1$ as the composition of paths  from $0$ to ${1\over 2}$ and then from ${1\over 2}$ to $1$, and use the composition of paths formula. 
% One way to improve the (rather poor)
%convergence of $(\ref{introzetadef})$ is to use the path composition formula and integrate first from $0$ to ${1 \over 2}$ and from ${1 \over 2}$ to $1$. More precisely, for $a_1,\ldots, a_n\in %\{0,1\}$ with $a_1=1$, $a_n=0$,
% \begin{eqnarray} 
 %I(0;a_1,\ldots, a_n;1)  &=  & \sum_{i=0}^n I(0;a_1,\ldots, a_i;\textstyle{1\over 2}) I(\textstyle{1\over 2}, a_{i+1},\ldots, a_n, 1) \nonumber \\
 %&  = & \sum_{i=0}^n I(0;a_1,\ldots, a_i;\textstyle{1\over 2}) I(0,1-a_{n},\ldots, 1-a_{i+1}, \textstyle{1\over 2}) \nonumber \end{eqnarray}
% where the second line follows from the first by applying $t\mapsto 1-t$ and the path reversal formula. 
The upshot is that every multiple zeta can be written in terms of  multiple polylogarithms evaluated at ${1\over 2}$. Many other methods are  also available. 
\vspace{0.05in}

\emph{iii).}
 This is only an algorithm in the true sense of the word in so far as it is possible to compute the coefficients $c_{\xi}$ $(\ref{cncoeff})$, and this  is the only transcendental input.
 A different realization of the  motivic multiple zeta values  (say in the $p$-adic setting, or otherwise) might lead to an exact algorithm for the computation of these coefficients too.
We hope that one can give  a theoretical upper bound for the prime powers which can occur in the denominators $c_{\xi}$ as a function of the weight (and choice of basis).
\vspace{0.05in}

\emph{iv).}  There is in fact no reason to suppose that our basis is an algebra basis, nor that it contains the depth one elements $\zetam(2n+1)$.
For example,  in \cite{Br23} we proved that the Hoffman elements:
\begin{equation} \label{Hoffbasis} \zetam(n_1,\ldots, n_r)\ ,  \hbox{ where } n_i \in 2, 3
\end{equation} 
are a vector space basis for $\Ho$. It is obvious that the number of such elements in weight $N$ is given by the integers $d_N$ of $(\ref{enumeration})$. 
This choice of basis gives a canonical map
$$\phi: \Ho \overset{\sim}{\To} \U$$
which respects the coactions,  maps $\zetam(2,\ldots, 2)$ ($n$ two's) to  $ {(6f_2)^n \over (2n+1)!}$ for all $n\geq 1$, and  for all $n=a+b+1$ satisfies
\begin{eqnarray}
c^{\phi}_{2n+1} \zetam(\underbrace{2,\ldots, 2}_a, 3, \underbrace{2,\ldots, 2}_b) & = &2 (-1)^{n} \Big( \binom{2n}{2a+2} - (1-2^{-2n}) \binom{2n}{2b+1}\Big)  \nonumber \\
c^{\phi}_{2n+1} \zetam(n_1,\ldots, n_r)& = &0 \hbox{ if at least } 2 \,\, n_i\hbox{'s  are equal to } 3 \nonumber
\end{eqnarray}  
  A slight variant of the previous algorithm allows one to decompose motivic MZV's into this basis also.
 \vspace{0.05in}
 
 \emph{v).}  A similar version of this algorithm also works for multiple polylogarithms evaluated at $N^{\mathrm{th}}$ roots of unity, in particular in the case of Euler sums ($N=2$). In some cases an explicit basis for the motivic iterated integrals at roots of unity is known by \cite{D}.
 \vspace{0.05in}

 \emph{vi).} Given a relation between motivic multiple zeta values, one can define operators $\partial^{\phi}_{2n+1}$ (for some choice of $\phi$), to obtain more relations of lower weight. 
  Applying the period map gives a relation between real MZVs.
 Thus a relation between motivic MZVs gives rise to a family of relations between real MZVs.
 
The converse is also true: the decomposition algorithm allows one to prove  an identity between motivic MZVs if one knows  sufficiently many relations between real MZVs
to determine all the coefficients $(\ref{cncoeff})$ which arise in the algorithm. This was alluded  to in point $(2)$ of the introduction.  In (\cite{Br23}, \S4) this idea was used to lift an identity between real MZV's   to the motivic level (it is in fact the definition of the motivic MZV's).

\section{ Worked example of the decomposition algorithm} \label{sectAppendix}
 We  use the following set  of motivic multiple zeta values as our independent algebra generators up to weight $10$ (compare the tables in \S\ref{sectMZVstructure}): 
 \begin{equation}\label{Appbasis} B=\{\zetam(2), \zetam(3) , \zetam(5), \zetam(7), \zetam(3,5), \zetam(9), \zetam(3,7)\}\ .
 \end{equation}
 We first associate to each element of $B$  an element in $\U$. To economize on notations, we denote $\partial^{\phi_B}_{\cdot}$ by $\partial_{\cdot}$, since 
 there is no confusion. 
 
\subsection{Construction of the basis polynomials} The elements  $\phi^B(b)\in \U$, for  $b\in B$, are defined as follows.   Firstly, 
$$\phi^B(\zetam(n))= f_n, \hbox{ for } n=2, 3,5,7,9\ , $$
by $(2)$ of \S\ref{sectdecompalg}. By direct application of definition \ref{Drdef}  we have:
{ \small \begin{eqnarray}
 D_3  \zetam(3,5)  &=  & \Imotl(0;100;1)\otimes \Imot(0;10000;1) + \Imotl(1;001;0)\otimes \Imot(0;10000;1)  \nonumber  \\
 D_5  \zetam(3,5)  &=  & \Imotl(1;00100;0)\otimes \Imot(0;100;1) + \Imotl(0;10000;1)\otimes \Imot(0;100;1)  \nonumber  
 \end{eqnarray}}
\noindent 
By  $(\ref{Idepth1shuff})$,  $\partial_3 \, \zetam(3,5) = 0$, $\partial_5\,  \zetam(3,5) = -5 \, \zetam(3)$, and therefore
\begin{equation}\label{App35} \phi^B(\zetam(3,5)) = -5 f_5 f_3\ , \end{equation}
following the prescription of $(2)$, \S\ref{sectdecompalg}. Similarly,
{ \small \begin{eqnarray}
 D_3 \zetam(3,7)  &=  & \Imotl(0;100;1)\otimes \Imot(0;1000000;1) + \Imotl(1;001;0)\otimes \Imot(0;1000000;1)  \nonumber  \\
 D_5  \zetam(3,7)  &=  & \Imotl(1;00100;0)\otimes \Imot(0;10000;1)   \nonumber   \\
 D_7  \zetam(3,7)  &=  & \Imotl(1;0010000;0)\otimes \Imot(0;100;1)+\Imotl(0;1000000;1)\otimes \Imot(0;100;1)   \nonumber   
  \end{eqnarray}}
 \noindent Thus  $\partial_3 \, \zetam(3,7)= 0$, $\partial_5 \, \zetam(3,7)=-6 \,\zetam(5)$, $\partial_7 \, \zetam(3,7)= -14 \,\zetam(3)$, \emph{i.e.}, 
  \begin{equation}\label{App37} \phi^B(\zetam(3,7)) = -14 f_7f_3 - 6 f_5f_5\ . \end{equation}
  This computation  proves that $B$ is indeed an algebra basis, since  the elements  in $\phi^B(\langle B\rangle_n)$  for $n\leq 10$
  are linearly independent.
  For  example, in weight 10 one checks that we have the following basis for $\U_{10}$:
   $$ f_2^5\ , \ f_3\sha f_3f_2^2 \  , \ f_3\sha f_5f_2 \ , f_5\sha f_5 \ , \ -5 f_5f_3f_2\ , \  f_3 \sha f_7 \ ,   -14f_7f_3 -6f_5f_5$$
 Therefore any motivic MZV of weight $10$ can be uniquely written
\begin{eqnarray}
\xi = a_0 \zetam(2)^5 +a_1\zetam(2)^2 \zetam(3)^2 + a_2 \zetam(2)\zetam(3) \zetam(5) + a_3 \zetam(5)^2 \nonumber \\ 
 + a_4 \zetam(2)\zetam(3,5) + a_5 \zetam(3)\zetam(7) + a_6 \zetam(3,7)\ , \label{wttenform}
 \end{eqnarray}
where $a_0,\ldots, a_6\in \Q$. From the   action of  $\partial_{3}, \partial_5, \partial_7$ computed in  $(\ref{App35})$, $(\ref{App37})$,  we see that the $a_i$ are  given  by applying the following operators
\begin{equation} \label{Appwt10operators}
a_1 =  {1\over 2}c^2_2\partial_3^2  \ ,\   a_2= c_2\partial_5\partial_3 \ , a_3 = {1\over 2} \partial_5^2 + {6 \over 14} [\partial_7,\partial_3] \end{equation}
$$ a_4 = {1\over 5} c_2 [\partial_3,\partial_5] \ , \ a_5 = \partial_7\partial_3 \ , \ a_6 = {1\over 14} [\partial_7,\partial_3]$$
to the element $\phi^B(\xi)$, where $c_2^n$ means taking the coefficient of $f_2^n$.

\subsection{Sample decompositions} Let us compute $\zetam(4,3,3)$ as a polynomial in our basis $B$.  From the  calculations $(4)$ below, we shall see that its  non-trivial sub and quotient sequences  are 
$\zetam(3,4)$,  $\zetam(4,3)$, $\zetam(2,3)$.
Working backwards, we decompose these elements in increasing order of weight.

\begin{enumerate} 
\item \emph{Decomposition of $\zetam(2,3)$}. By example \ref{ex2343},  $ \partial_3 \zetam(2,3) = 3\, \zetam(2)$.   In weight five, $\U_5\cong \Q  f_3f_2  \oplus \Q f_5$, so  it follows that $\zetam(2,3)$  is of the form $c \,\zetam(5) + 3\, \zetam(3)\zetam(2)$, where $c\in \Q$. By  numerical  computation, or some other method, we check that: 
$$c = {\zeta(2,3) - 3\, \zeta(2)\zeta(3) \over \zeta(5) } \sim -{11\over 2}\ .$$
Thus $\zetam(2,3) =  -{11\over 2} \zetam(5)+3 \zetam(3) \zetam(2).$ 
\vspace{0.05in}

\item \emph{Decomposition of $\zetam(4,3)$}. By example \ref{ex2343},   we have  $\partial_3 \zetam(4,3) =  \zetam(4) = {2\over 5} \zetam(2)^2$, and 
$\partial_5 \zetam(4,3) = 10 \zetam(2)$.  In weight 7, 
$$\U_7 \cong \Q  f_3 f^2_2 \oplus f_5f_2 \oplus \Q f_7$$
so $\phi^B(\zetam(4,3))$ is of the form $c f_7 + 10 f_5 f_2 + {2\over 5} f_3 f^2_2 $. By numerical computation  or otherwise, 
$$c = {\zeta(4,3) - 10\, \zeta(2)\zeta(5) -{2\over 5} \zeta(3) \zeta(2)^2 \over \zeta(7) } \sim -18\ .$$
Thus $\zetam(4,3) = -18\, \zetam(7) + 10 \, \zetam(5) \zetam(2) + {2\over 5} \,  \zetam(3) \zetam(2)^2.$
\vspace{0.05in}
\item  \emph{Decomposition of $\zetam(3,4)$}.  We omit the computation, which is similar, and merely state that
$\zetam(3,4)  = 17 \zetam(7) - 10\,\zetam(5) \zetam(2).$ (It also follows immediately  from $(2)$ and the so-called  stuffle relation $\zetam(3)\zetam(4)=\zetam(3,4)+\zetam(4,3)+\zetam(7)$.)

  \vspace{0.05in}
\item \emph{Decomposition of $\zetam(4,3,3)$}.  By $(\ref{mainformula})$ and lemma \ref{lemImotrelations},
{\small \begin{eqnarray}
D_3 \zetam(4,3,3) & = & (\Imotl(0;100;1) + \Imotl(1;001;0) + \Imotl(0;100;1)) \otimes \Imot(0;1000100;1)  \nonumber \\
  & = & \zetal(3) \otimes \zetam(3,4) \ . \nonumber \\
D_5 \zetam(4,3,3) & = & \Imotl(1;00010;0)\otimes \Imot(0;10100;1)  +\Imotl(0;00100;1)\otimes \Imot(0;10100;1)  \nonumber \\
  & = & 10\, \zetal(5) \otimes \zetam(3,2) \ . \nonumber \\
D_7 \zetam(4,3,3) & = & (\Imotl(1;1000100;0)+ \Imotl(1;0001001;0) + \Imotl(0;0100100;1)) \otimes \Imot(0;100;1)  \nonumber \\
  & = & (\zetal(4,3) -\zetal(3,4) - 3(\zetal(4,3)+\zetal(3,4)   )\otimes \zetam(3) \  ,\nonumber \\
  & = & -32 \, \zetal(7) \otimes \zetam(3)  \nonumber
  \end{eqnarray} }
 Thus we have:
  {\small \begin{eqnarray}
\phi^B(\partial_3  \zetam(4,3,3))  = \phi^B( \zetam(3,4)) & =  &-18 f_7 + 10 f_5f_2  +{2\over 5}f_3f_2^2 \ . \nonumber \\
\phi^B( \partial_5 \zetam(4,3,3))  =  10 \,\phi^B(\zetam(3,2)) &  = & -55 f_5 + 30  f_3 f_2 \ . \nonumber \\
 \phi^B(\partial_7 \zetam(4,3,3))  =   -32\, \phi^B(\zetam(3))&  =  & -32 f_3\ . \nonumber 
  \end{eqnarray} }
 Using the  equations ($\ref{Appwt10operators})$ we conclude that  
$$\zetam(4,3,3) = a_0\,  \zetam(2)^5 + {1\over 5} \zetam(2)^2 \zetam(3)^2 + 10 \, \zetam(2) \zetam(3) \zetam(5) - {49 \over 2} \zetam(5)^2$$
$$ \qquad \qquad - 18 \, \zetam(3)\zetam(7) - 4\,  \zetam(2) \zetam(3,5)  +\zetam(3,7)$$
Finally, by numerical computation, one checks once again that
$$\zeta(4,3,3)- \Big[{1\over 5} \zeta(2)^2\zeta(3)^2 + \ldots + \zeta(3,7) \Big]   \sim {271 \over 10} \zeta(10) =  {4336 \over 1925}\zeta(2)^5$$
\end{enumerate}
which gives the coefficient $a_0$ of $\zetam(2)^5$.  In this example  the coefficients $a_1,a_2,a_4$ of $(\ref{wttenform})$ are computed exactly; 
the others are obtained indirectly via the period map and numerical approximation.

\section{Acknowledgements}
Very many thanks to Pierre Cartier for a thorough reading of the text and many detailed corrections and comments.  
This paper was completed during a stay at the Research Institute for Mathematical Sciences, of  Kyoto University, 
and based on a talk given at the conference: ``Development of Galois -Teichm\"uller Theory and Anabelian Geometry'' held there.
 I would like to   thank the organisers heartily  for their hospitality.

 This work was  supported by  European Research Council grant no.  257638: `Periods in algebraic geometry and physics'. 

\bibliographystyle{plain}
\bibliography{main}

\end{document}